\documentclass{article}

\usepackage{tabularx} % extra features for tabular environment
\usepackage{amsmath}  % improve math presentation
\usepackage{amssymb}
\usepackage{amsthm}
\usepackage{stmaryrd}
\usepackage{graphicx} % takes care of graphic including machinery
\usepackage[usenames, dvipsnames]{xcolor}
\usepackage{tikz}
\usetikzlibrary{calc}
\usetikzlibrary{positioning}
\usepackage{enumitem}
\usepackage{breqn}

\newtheorem{thm}{Theorem}
\newtheorem{lem}[thm]{Lemma}
\newtheorem{claim}[thm]{Claim}
\newtheorem{obs}[thm]{Observation}
\newtheorem{ques}[thm]{Question}
\newtheorem{prop}[thm]{Proposition}

\theoremstyle{definition}
\newtheorem{defi}[thm]{Definition}

\newcommand{\NN}{\mathbb{N}}
\newcommand{\set}[1]{\left\{#1\right\}}

\newcommand{\abs}[1]{\left|#1\right|}

\newcommand{\ignore}[1]{}

\newcommand{\ie}{{\it i.e.\ }}

\newcommand{\rc}{r.c.\ }
\newcommand{\resp}{resp.\ }
\usepackage{subfig}

\newcommand{\tec}{$2$-edge-colored}

\newcommand{\setk}[1]{\left\llbracket #1 \right\rrbracket}

\newcommand{\PROBLEM}[3]{
\begin{description}
\itemsep -6mm
\item[\mbox{\hskip .6cm}Problem:] {\sc #1}\\
\item[\mbox{\hskip .6cm}Instance:] #2\\
\item[\mbox{\hskip .6cm}Question:] #3
\end{description}
}

\newcommand{\tikzstylemacro}{
\tikzstyle{n} = [draw,circle,minimum size=0.8mm,inner sep=0pt,outer sep=0pt,fill=black];
\tikzstyle{g} = [line width=1pt,color=NavyBlue];
\tikzstyle{r} = [line width=1pt, color=red,dashdotted];
}
\begin{document}

%\title{Complexity of the achromatic number of signed graphs}
\title{On the achromatic number of signed graphs}
\author{Dimitri Lajou\footnote{\'ENS Lyon and
Univ. Bordeaux, Bordeaux INP, CNRS, LaBRI, UMR5800, F-33400 Talence, France.
Partially supported by the ANR project HOSIGRA (ANR-17-CE40-0022)}}
\date{\today}

\maketitle

\begin{abstract}
In this paper, we generalize the concept of complete coloring and
achromatic number to \tec{} graphs and signed graphs. %After showing
We give
some useful relationships between %our different definitions
different possible definitions of such achromatic numbers
%we
and prove that computing any of them is NP-complete.
%and that
%while the problem becomes FPT when the parameter is the threshold that we need to
%decide if it is smaller or greater than or number.
%Moreover, we prove that for any integer $k$,
%deciding whether any of these numbers is at least $k$ is FPT when the parameter is $k$.
\end{abstract}

%\begin{keyword}
%	Signed graph \sep Two-edge colored graph \sep Achromatic number \sep NP-Complete problem
%\end{keyword}

\section{Introduction}

All the graphs we consider are undirected and simple.
We denote by $V(G)$ and $E(G)$ the set of vertices and the set of edges of a graph $G$, respectively. The {\em neighbourhood} $N_G(u)$ of a vertex $u$ of $G$ is the set of vertices which share an edge with $u$ in $G$.
A {\em \tec{} graph} is a graph where each edge of $E(G)$ can be either positive or negative.
We denote by $(G,C)$ such a graph, where $G$ is an ordinary graph, called the {\em underlying graph} of $(G,C)$, and $C$ is the set of negative edges, also called the {\em signature} of $(G,C)$.
%, also called the \emph{signature} of $(G,C)$.

\medskip

A {\em signed graph} $[G,\Sigma]$ is an equivalence class on the set of \tec{} graph $(G,C)$ where $C \subseteq E(G)$.
Two \tec{} graphs $(G,C)$ and $(G,D)$ are {\em equivalent} if we can go from one to the other by a series of re-signings, where {\em re-signing} at a vertex $v$ consists in inverting the sign of all the edges incident with~$v$. 
A {\em representative} of a signed graph $[G,\Sigma]$ is a \tec{} graph which belongs to $[G,\Sigma]$.
A {\em signature} of a signed graph $[G,\Sigma]$ is a signature of one of its representatives.
For the rest of this article, we will use the adjective {\em \tec{}} when the signature is fixed, and {\em signed} when re-signing
is allowed.

We write $(G,C) \in [G,\Sigma]$ if $(G,C)$ is one of the representatives of $[G,\Sigma]$.  The {\em canonical representative} of $[G,\Sigma]$ is the \tec{} graph $(G,C)$ where $C = \Sigma$. Note that if $(G,C) \in [G,\Sigma]$ and $(G,C') \in [G,\Sigma]$, then 
%$[G,\Sigma]$, 
$[G,\Sigma_C]$ and $[G,\Sigma_{C'}]$, where $\Sigma_C = C$ and $\Sigma_{C'} = C'$, are both equal to the signed graph $[G,\Sigma]$. 
%This implies that the canonical representative of $[G,\Sigma]$ does not only depend on the signed graph, but also on the signature used in its notation.

To avoid any possible confusion, signatures of \tec{} graphs will be denoted by Roman letters
while signatures of signed graphs will be denoted by Greek letters.
%For the notation we will always denote the set of red edges of a \tec{} graph by a Roman letter and the set of negative edges for a signed graph as a Greek letter to differentiate between the two type of graph.

%\medskip
%
%The {\em alternating path} of length $2$ is the path of length $2$ with the two types of edges (red and blue or positive and negative).
%The {\em unbalanced cycle} of length $k\ge 3$, denoted $UC_k$, is the
%signed cycle of length $k$ having an odd number of negative edges.
%It is not difficult to observe that all such signatures are equivalent.
%{\textcolor{red}{Pas utilisé dans l'intro... à déplacer ailleurs~?}}

\medskip

A {\em $k$-(vertex-)coloring} of a \tec{} graph is a function $\alpha:V(G) \rightarrow \setk{k}$,
where $\setk{k}$ denotes the set $\{1,\dots,k\}$,
such that $\alpha(u) \neq \alpha(v)$ for every $uv \in E(G)$ and,
for every two colors $i$ and $j$, all the edges $uv$ with $\alpha(u)=i$ and $\alpha(v)=j$ have the same sign.
The {\em chromatic number} $\chi_2(G,C)$ of a \tec{} graph $(G,C)$
is the smallest $k$ for which $(G,C)$ admits a $k$-coloring.
Similarly, the {\em chromatic number} $\chi_s[G,\Sigma]$ of a signed graph $[G,\Sigma]$
is the smallest $k$ for which $[G,\Sigma]$ admits a representative $(G,C)$ with $\chi_2(G,C)=k$. Alternatively, a {\em $k$-(vertex-)coloring} of a signed graph is a $k$-(vertex-)coloring of one of its representative and $\chi_s[G,\Sigma]$ is the smallest $k$ for which a $k$-coloring of $[G,\Sigma]$ exists.

A {\em (\tec) homomorphism} of a \tec{} graph $(G,C)$ to a \tec{} graph $(H,D)$
is a function $\varphi$ from $V(G)$ to $V(H)$ such that,
for every pair of vertices $u,v \in V(G)$, $uv\in E(G)$ implies $\varphi(u)\varphi(v)\in E(H)$ and, for every edge $uv\in E(G)$,
$uv\in C$ if and only if $\varphi(u)\varphi(v)\in D$.
Similarly,
a {\em (signed) homomorphism} of a signed graph $[G,\Sigma]$ to a signed graph $[H,\Pi]$
is a function $\varphi$ from $V(G)$ to $V(H)$
which is a homomorphism of a \tec{} graph $(G,C)$ to a \tec{} graph $(H,D)$,
where $(G,C)$ is a representative of $[G,\Sigma]$ and $(H,D)$ is a representative of $[H,\Pi]$.
As stated in \cite{Naserasr2014}, we can observe that we can always choose $D=\Pi$, so that re-signing is done only on $[G,\Sigma]$, if needed. 
A signed homomorphism of $[G,\Sigma]$ to $[H,\Pi]$ can thus be viewed as 
a \tec{} homomorphism of $(G,C)$ to $(H,D_\Pi)$, where $(G,C)$ is a representative of $[G,\Sigma]$ (obtained
by re-signing $[G,\Sigma]$) and $D_\Pi=\Pi$.
%is equivalent to being the composition of an arbitrary number of \tec{} homomorphism and re-signings.
%An homomorphism of \tec{} graph (\resp signed graph) is a function $\varphi$ from $(G,C)$ to $(H,S)$ (\resp $(G,\Sigma)$ to $(H,\Pi)$) which maps vertices to vertices such that each edge of $(G,C)$ map to an edge of $(H,S)$ of the same color (reps. which maps one representative of the domain to one representative of the co-domain as a \tec{} homomorphism).
Homomorphisms of \tec{} graphs were introduced by Alon and Marshall in~\cite{Alon1998},
while homomorphisms of signed graphs were introduced by Naserasr, Rollov\'a and Sopena in~\cite{Naserasr2014}.

A {\em surjective homomorphism} is a homomorphism whose co-domain is the image of its domain.
With each coloring of a \tec{} graph,
we can associate a surjective \tec{} homomorphism which identifies all vertices having the same color. 
Similarly, with any coloring of a signed graph, we can associate a signed homomorphism which re-signs the signed graph to get the signature for which the coloring is defined, and then identifies all vertices having the same color.

An {\em unbalanced path} of order $k\ge 2$ in a \tec{} graph, denoted $UP_k$, is a path of order $k$ having an odd number of negative edges.
A {\em balanced path} of order $k\ge 2$  in a \tec{} graph, denoted $BP_k$, is a path of order $k$ having an even number of negative edges.
An {\em unbalanced cycle} of length $k\ge 3$  in a \tec{} graph, denoted $UC_k$, is a
cycle of length $k$ having an odd number of negative edges.
An {\em balanced cycle} of length $k\ge 3$  in a \tec{} graph, denoted $BC_k$, is a
cycle of length $k$ having an even number of negative edges.
Note that in a signed graph, whether a cycle is balanced or unbalanced does not depend on the representative of the signed graph (\ie this structure is invariant by re-signing).
In fact, Zaslavsky in \cite{Zaslavsky1982} showed that a signed graph is entirely characterized by its underlying graph and the set of its balanced cycles (or the set of its unbalanced cycles).

In what follows, a {\em digon} will be a $UC_2$, \ie two vertices linked by two edges, one positive and one negative.
As our graphs are simple, we want to make sure that they contain no loops and no digons. 
This will become particularly important when we construct homomorphisms, as the image graph must be simple.
For a \tec{} graph this means, in particular, that we cannot identify vertices which belongs to the same edge or to the same $UP_3$, as identifying them would create a loop in the first case and a digon in the second case. 
Note that we can always identify a pair of vertices that do not belong to the same edge or to the same $UP_3$.
Any two such vertices are said to be {\em identifiable}.
%When two vertices are either part of the same edge or part of the same $UP_3$, then we say that they are {\em not identifiable}. 
%They are {\em identifiable} otherwise. 
%
In a signed graph, before identifying two non-adjacent vertices $u$ and $v$, we thus need to re-sign the signed graph in order to remove every $UP_3$ containing $u$ and $v$.
Note that this is not always possible. For example, in the unbalanced cycle $UC_4$, we cannot identify any pair of vertices.
Indeed, Naserasr, Rollov\'{a} and Sopena showed in~\cite{Naserasr2014} that two vertices are {\em  identifiable} 
if and only if they do not belong to the same edge or to the same $UC_4$. 
%We say that they are {\em identifiable} otherwise.

Note that we can construct a surjective homomorphism of a \tec{} graph or a signed graph by repeatedly identifying pairs of identifiable vertices, until no such pair exists. The image graph is then the obtained \tec{} graph or signed graph.
%When two vertices cannot be assigned the same color, we say that they are {\em not identifiable}.

%We note that a coloring is proper if and only if the image graph doesn't contains loops nor digon.

We will only consider surjective homomorphisms in the rest of this paper, and write
$(G,C)\rightarrow_2 (H,D)$ (\resp $[G,\Sigma]\rightarrow_s [H,\Pi]$) whenever there exists a surjective homomorphism
of $(G,C)$ to $(H,D)$ (\resp of $[G,\Sigma]$ to $[H,\Pi]$).

A {\em \tec{} clique} (\resp a {\em signed clique}) is a \tec{} graph (\resp a signed graph) which is its unique homomorphic image,
up to isomorphism.
Alternatively, it can be defined as a graph whose chromatic number equals its order.
The chromatic number can thus be seen as the smallest order of a clique to which the graph maps by a surjective homomorphism.
Note here that every signed clique is also a \tec{} clique.

In Figure~\ref{ex:cliques}, we represent two \tec{} graphs that are \tec{} cliques but whose signed versions are not signed cliques. Indeed, in Figure~\ref{2ec-clique1}, it suffices to re-sign one of the degree~$1$ vertices to be able to identify them. In Figure~\ref{2ec-clique2}, it suffices to re-sign the vertices $a$ and $b$, or $c$ and $d$, to get a pair of identifiable vertices.

\begin{figure}%[t]
  \centering
  \subfloat[A 2-edge colored clique which is not a signed clique: the graph $UP_3$.]{
  \begin{tikzpicture}
    \tikzstylemacro{}

    \node[n] (a) at (-1,0) {};
    \node[n] (b) at (0,1) {};
    \node[n] (c) at (1,0) {};
    \draw (b) edge[r] (a) edge[g] (c);

  \end{tikzpicture}
  \label{2ec-clique1}
 }
  \hskip 1.5cm
  \subfloat[A 2-edge colored clique which is not a signed clique.]{
  \begin{tikzpicture}
    \tikzstylemacro{}

    \node[n,label=right:$a$] (a) at (30:1) {};
    \node[n,label=right:$b$] (b) at (-30:1) {};
    \node[n,label=left:$c$] (c) at (150:1) {};
    \node[n,label=left:$d$] (d) at (210:1) {};
    \node[n,label=above:$x$] (x) at (0,0) {};

	\draw (x) edge[g] (a) edge[g] (b) edge[r] (c) edge[r] (d);
    \draw (a) edge[g] (b);
    \draw (d) edge[g] (c);

  \end{tikzpicture}
  \label{2ec-clique2}
 }
  \hskip 1.5cm
 \subfloat[The signed clique $UC_4$.]{
  \begin{tikzpicture}
    \tikzstylemacro{}

    \node[n] (a1) at (-2,0.5) {};
    \node[n] (a2) at (-2,-0.5) {};
    \node[n] (a3) at (-1,-0.5) {};
    \node[n] (a4) at (-1,0.5) {};
    
	\draw (a1) edge[g] (a2) edge[r] (a4);
	\draw (a3) edge[g] (a2) edge[g] (a4);

  \end{tikzpicture}
  \label{signed-clique1}
}
 \hskip 1.5cm
 \subfloat[A signed clique.]{
  \begin{tikzpicture}
    \tikzstylemacro{}

    \node[n] (a1) at (-2,0.5) {};
    \node[n] (a2) at (-2,-0.5) {};
    \node[n] (a3) at (-1,-0.5) {};
    \node[n] (a4) at (-1,0.5) {};
    
    \node[n] (b1) at (2,0.5) {};
    \node[n] (b2) at (2,-0.5) {};
    \node[n] (b3) at (1,-0.5) {};
    \node[n] (b4) at (1,0.5) {};
    
    \node[n] (x1) at (0,1) {};
    \node[n] (x2) at (0,-1) {};
    
    \draw (a1) edge[g] (a2) edge[r] (a4);
	\draw (a3) edge[g] (a2) edge[g] (a4);
	
	\draw (b1) edge[g] (b2) edge[r] (b4);
	\draw (b3) edge[g] (b2) edge[g] (b4);
	
	\foreach \i in {1,2,3,4}{
		\draw (x1) edge[r] (a\i);
		\draw (x1) edge[g] (b\i);
		\draw (x2) edge[g] (a\i);
		\draw (x2) edge[g] (b\i);
	}

  \end{tikzpicture}
  \label{signed-clique2}
}
\caption{Two \tec{} cliques and two signed cliques.}
\label{ex:cliques}
\end{figure}
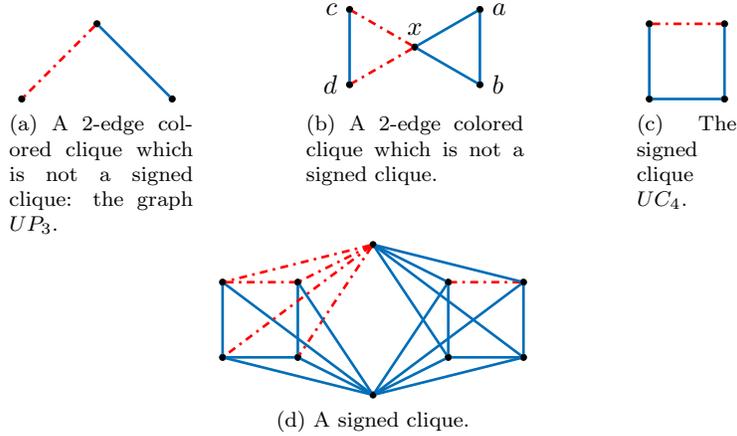

\medskip

%SOPENA

Harary and Hedetniemi defined in~\cite{Harary1970} the {\em achromatic number} $\psi(G)$
of a graph $G$ as the largest $k$ such that there exists a {\em complete $k$-coloring} of $G$,
where a complete $k$-coloring is a coloring where each pair of colors appears on some edge
(see also~\cite[Chapter 12]{Chartrand}).
Therefore, a complete coloring is nothing but a surjective homomorphism to a clique, where a {\em homomorphism} 
is an edge-preserving vertex mapping.
%is a function that associate the vertices of the graph to those of the image graph and conserve the adjacency.
%{\it I.e.} the image graph is a clique.

Similarly, we define the {\em (\tec{}) achromatic number} $\psi_2(G,C)$
of a \tec{} graph $(G,C)$ as the largest order of a \tec{} clique $(K,D)$
such that $(G,C)\rightarrow_2 (K,D)$,
and the {\em (signed) achromatic number} $\psi_s[G,\Sigma]$
of a signed graph $[G,\Sigma]$ as the largest order of a signed clique $[K,\Pi]$
such that $[G,\Sigma]\rightarrow_s [K,\Pi]$
 (recall that all homomorphisms we consider
are surjective).

A {\em complete $k$-coloring} of a \tec{} graph $(G,C)$, or of a signed graph $[G,\Sigma]$,
can thus be defined as a $k$-coloring such that, for every two colors $i$ and $j$, there exist an $i$-colored
vertex and a $j$-colored vertex which are not identifiable.
In such a case, we say that the colors $i$ and~$j$ are {\em in conflict}.

%We define $\psi_2$ as the achromatic number of a \tec{} graph as the biggest clique reachable by surjective \tec{} homomorphism from our graph. Alternatively it is the biggest complete coloring on the graph. Here a complete coloring is defined as a proper coloring where each pair of color have at least two vertices that are not identifiable in the image graph associated to the coloring.
%In such case we say that the two colors are in conflict.

%For a signed graph we can define $\psi_s$ as the biggest signed clique reachable by signed homomorphism.
The notions of complete colorings and achromatic numbers have also been extended to digraphs in \cite{AMOR}, \cite{EDWARDS} or~\cite{NEUMANNLARA},
and to oriented graphs in~\cite{SOP}.
However, the situation is fundamentally different since any two colored or signed edges $uv$ and $vu$ are identical,
while any two arcs $\overrightarrow{uv}$ and $\overrightarrow{vu}$ are not.

\medskip

In this paper, we are mainly interested in the three following questions.
\begin{enumerate}
\item For a given signed graph $[G,\Sigma]$, what can we say about the \tec{} achromatic number
of $(G,C)$, for any signature $C$ being equivalent to $\Sigma$?
\item For a given graph $G$, what can we say about the \tec{} achromatic number
of $(G,C)$, for any signature $C$?
\item For a given graph $G$, what can we say about the signed achromatic number
of $[G,\Sigma]$, for any signature $\Sigma$?
\end{enumerate}

To this end, we define the six following types of achromatic numbers.

\begin{defi}
For any graph $G$ and any signed graph $(G,\Sigma)$, we let
\begin{itemize}
  \item $\psi_{\max}[G,\Sigma] := \max \set{\psi_2(G,C) | (G,C) \in [G,\Sigma] }$,  the
  {\em signed max-achromatic number} of $[G,\Sigma]$,
  \item $\psi_{\min}[G,\Sigma] := \min \set{\psi_2(G,C) | (G,C) \in [G,\Sigma] }$,  the
  {\em signed min-achromatic number} of $[G,\Sigma]$,
  \item $\psi_{\max}(G) := \max \set{\psi_2(G,C) |  C \subseteq E(G) }$,  the
  {\em \tec{} max-achromatic number} of~$G$,
  \item $\psi_{\min}(G) := \min \set{\psi_2(G,C) |  C \subseteq E(G)  }$,  the
  {\em \tec{} min-achromatic number} of~$G$,
  \item $\psi_{\max}^{signed}(G) := \max \set{\psi_s[G,\Sigma] |  \Sigma \subseteq E(G) }$,  the
  {\em signed max-achromatic number} of~$G$,
  \item $\psi_{\min}^{signed}(G) := \min \set{\psi_s[G,\Sigma] |  \Sigma \subseteq E(G)  }$,  the
  {\em signed min-achromatic number} of~$G$.
  \end{itemize}
\end{defi}

%The three most important are the achromatic number of a \tec{} graph $\psi_2$, the achromatic number of a signed graph $\psi_s$ and the max achromatic number of a signed graph $\psi_{\max}$.

We will study the complexity status of the problem of determining each
of these numbers.
Our paper is organized as follows.
In the next section we detail some properties of these numbers and state our main results.
Section~\ref{section-NP} is devoted to the proofs of these results and we propose directions
for future research in Section~\ref{section-discussion}.

%\medskip
%

%{\textcolor{red}{Pas utilisé dans l'intro... à déplacer ailleurs~?}}

\section{{Preliminaries and statement of results}}

In this section, we detail some properties of the achromatic numbers
introduced in the previous section.
We first compare chromatic and achromatic numbers of \tec{} graphs
and signed graphs.
%we show some relationship between the chromatic number and the three important definition of achromatic number.

\begin{thm}
  For every signed graph $[G,\Sigma]$ and every \tec{} graph $(G,C) \in [G,\Sigma]$,
  $$ \chi_2(G,C) \leq \psi_2(G,C) \leq \psi_{\max}[G,\Sigma].$$
\end{thm}

\begin{proof}
By definition, the chromatic number of a \tec{} graph is at most its achromatic number.
Since $\psi_{\max}[G,\Sigma]$ is the maximum value of $\psi_2(G,C')$ taken over all $(G,C') \in [G,\Sigma]$, it is at least $\psi_2(G,C)$.
\end{proof}

\begin{thm}
  For every signed graph $[G,\Sigma]$,
  $$ \chi_s[G,\Sigma] \leq \psi_s[G,\Sigma] \leq \psi_{\max}[G,\Sigma].$$
\end{thm}

\begin{proof}
Again, by definition, the chromatic number of a signed graph is smaller than its achromatic number.
Since every signed clique is also a \tec{} clique, for every signed clique $[K,\Pi]$,
$[G,\Sigma] \rightarrow_s [K,\Pi]$ implies $(G,C) \rightarrow_2 (K,D_\Pi)$
for some \tec{} graph $(G,C) \in [G,\Sigma]$ and $D_\Pi = \Pi$, which in turn implies
$\psi_s[G,\Sigma] \leq \psi_2(G,C) \leq \psi_{\max}[G,\Sigma]$.
%
%  The maximum signed clique attained by $(G,\Sigma)$ is also a \tec{} clique thus there is one signature of $(G,\Sigma)$ such that its \tec{} achromatic number is greater than or equal to $\psi_s(G,\Sigma)$. So the maximum over all those \tec{} achromatic numbers is greater than or equal to $\psi_s(G,\Sigma)$.
\end{proof}

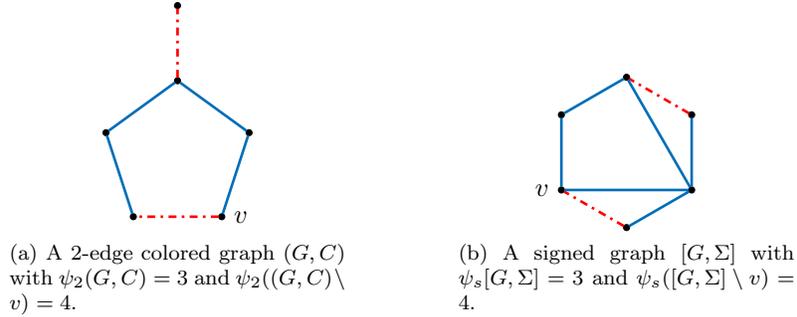
\begin{figure}%[t]
  \centering
  \subfloat[A 2-edge colored graph $(G,C)$ with $\psi_2(G,C) = 3$ and $\psi_2((G,C) \setminus v) = 4$.]{
  \begin{tikzpicture}
    \tikzstylemacro{}

    \node[n] (a) at (90:1) {};
    \node[n] (b) at (162:1) {};
    \node[n] (c) at (234:1) {};
    \node[n,label=right:$v$] (d) at (306:1) {};
    \node[n] (e) at (378:1) {};
    \node[n] (f) at (90:2) {};

    \node (i1) at (-2,0) {};
    \node (i2) at (2,0) {};

    \draw (a) edge[r] (f) edge[g] (e);
    \draw (b) edge[g] (a) edge[g] (c);
    \draw (d) edge[r] (c) edge[g] (e);

  \end{tikzpicture}
  \label{ct-ex-indu-tec}
 }
  \hskip 1.5cm
 \subfloat[A signed graph \text{$[G,\Sigma]$} with $\psi_s\text{$[G,\Sigma]$} = 3$ and $\psi_s(\text{$[G,\Sigma]$} \setminus v) = 4$.]{
  \begin{tikzpicture}
    \tikzstylemacro{}

    \node[n] (a) at (90:1) {};
    \node[n] (b) at (150:1) {};
    \node[n,label=left:$v$] (c) at (210:1) {};
    \node[n] (d) at (270:1) {};
    \node[n] (e) at (330:1) {};
    \node[n] (f) at (30:1) {};

    \node (i1) at (-2,0) {};
    \node (i2) at (2,0) {};

    \draw (f) edge[r] (a) edge[g] (e);
    \draw (b) edge[g] (a) edge[g] (c);
    \draw (d) edge[r] (c) edge[g] (e);
    \draw (e) edge[g] (a) edge[g] (c);

  \end{tikzpicture}
  \label{ct-ex-indu-signed}
}
\caption{Deleting a vertex in a \tec{} graph or in a signed graph.}
\label{ct-ex-indu}
\end{figure}

An interesting property of the achromatic number of
ordinary graphs is that for every graph $G$ and every vertex $v \in V(G)$, $\psi(G \setminus v) \leq \psi(v)$~\cite{Chartrand}.
However, this is no longer true for \tec{} graphs and signed graphs.
Indeed, one can check that removing the vertex $v$ in Figures~\ref{ct-ex-indu-tec} and~\ref{ct-ex-indu-signed}
increases the corresponding achromatic number by one.

This equality still holds for some particular vertices.
%provided that the deleted vertex satisfies some property.
Let $(G,C)$ be a \tec{} graph and $u$, $v$ be two vertices of $V(G)$.
We say that $u$ and $v$ are {\em twins} if $u$ and $v$ have the same colored neighborhood, \ie
for every vertex $w\in V(G)$, $uw\in E(G)$ if and only if $vw\in E(G)$ and
$uw\in C$ if and only if $vw\in C$.
In that case, we say that $v$ is a {\em twin vertex} of $u$.
Similarly, two vertices $u$ and $v$ of a signed graph $[G,\Sigma]$
are {\em twins} if they are twins in the \tec{} graph $(G,C_\Sigma)$ with $C_\Sigma:=\Sigma$, or if they are twins in the \tec{} graph obtained from $(G,C_\Sigma)$ after re-signing at $u$ (\ie up to re-signing one of them, they are twins in a representative of $[G,\Sigma]$). 

%We still have a way to remove vertices and keep some knowledge on the achromatic number by using the following theorem:

\begin{prop}
  \label{achrom-induced-classe}
For every \tec{} graph $(G,C)$ and every vertex $v\in V(G)$,
if $v$ has a twin vertex $u$, then $\psi_2((G,C)\setminus v)$ $\leq$ $\psi_2(G,C)$.
For every signed graph $[G,\Sigma]$ and every vertex $v\in V(G)$,
if $v$ has a twin vertex $u$, then
$\psi_s([G,\Sigma]\setminus v)$ $\leq$ $\psi_s[G,\Sigma]$
and $\psi_{\max}([G,\Sigma]\setminus v)$ $\leq$ $\psi_{\max}[G,\Sigma]$.
%  For a signed graph (\resp a \tec{} graph) $G$. Take $u,v \in G$  with the same neighborhood such that if $z$ is a neighbor of $u$ and $v$  then $uz$ and $vz$ have the same sign (\resp same color).
%  In such a case $\psi_s(G\setminus v)$ $\leq$ $\psi_s(G)$ and $\psi_{\max}(G\setminus v)$ $\leq$ $\psi_{\max}(G)$ (\resp $\psi_2(G\setminus v)$ $\leq$ $\psi_2(G)$).
\end{prop}

\begin{proof}
Let $\alpha$ be a complete coloring of $(G,C)\setminus v$ using $\psi_2((G,C)\setminus v)$ colors.
By setting $\alpha(v)=\alpha(u)$, $\alpha$ clearly extends to a complete coloring of $(G,C)$,
which implies $\psi_2((G,C)\setminus v)$ $\leq$ $\psi_2(G,C)$. 
The same argument clearly implies the two inequalities for signed graphs, after maybe re-signing at $v$ so that $u$ and $v$ are twins in the \tec{} representative of $[G,\Sigma]$.
%  Take $\alpha$ a coloring of $G \setminus v$ reaching the achromatic number (for signed graph we fix the corresponding signature). To show the theorem it suffices to show that $\alpha$ can be extended to be a coloring of $G$. We can identify $u$ and $v$ in $G$ to get the graph $G \setminus v$. Thus by taking the color of $v$ to be the same as $u$ the coloring create the same homomorphic image as the coloring of $G \setminus v$. Thus it is a proper coloring and $G$ as a complete coloring with exactly the achromatic number of $G \setminus v$ colors.
\end{proof}

\medskip

Before stating our results, we give some properties of the \tec{} and signed cliques as they are at the center
of the definition of all achromatic numbers we have introduced.

As observed before, any two vertices of a \tec{} clique or of a signed clique are
not identifiable.

\begin{obs}
\label{obs:2ecclique}
In a \tec{} clique $(K,D)$, every two vertices $u$ and $v$ satisfy at least one of the following:
  \begin{itemize}
    \item $uv \in E(G)$,
    \item $u$ and $v$ are end vertices of the same $UP_3$.
  \end{itemize}
This implies in particular that the diameter of $K$ is at most $2$.
\end{obs}

\begin{thm}[Naserasr, Rollov\'{a} and Sopena~\cite{Naserasr2014}]
In a signed clique $[K,\Pi]$, every two vertices $u$ and $v$ satisfy at least one of the following:
  \begin{itemize}
    \item $uv \in E(G)$,
    \item $u$ and $v$ are antipodal vertices of the same $UC_4$.
  \end{itemize}
\end{thm}

Let us remark that $\psi_2$ and $\psi_s$ can be viewed
as the largest order of a \tec{} clique or of a signed clique obtained by the following algorithm: while there exists two vertices identifiable, identify them. In the signed case, we may need to re-sign before identifying vertices.

We can construct a signed clique from a \tec{} clique by the following construction.

\begin{lem}
\label{2ectosignedclique}
	For a \tec{} graph $(K,D)$, if $(K',D')$ is the \tec{} graph obtained by adding one vertex $z$ to $(K,D)$ and adding, for every $u \in V(K)$, a positive edge $uz$, then $(K,D)$ is a \tec{} clique if and only if the signed graph $[K',\Sigma']$ where $\Sigma' = D'$ is a signed clique.
\end{lem}

\begin{proof}
Suppose first that $(K,D)$ is a \tec{} clique and
let $u$ and $v$ be any two vertices of $K'$. 
If $u=z$ or $v=z$, then there is an edge $uv$ by construction, so that $u$ and $v$ are not identifiable. 
Otherwise, both $u$ and $v$ belong to $K$. 
If there is no edge $uv$ in $K$, then $u$ and $v$ are the end vertices of some $UP_3$ of $(K,D)$, say $uwv$. 
Then, by construction, $uwvz$ is a $UC_4$ in $(K',D')$ and thus in $[K',\Sigma']$, so that $u$ and $v$ are not identifiable. 
This implies that $[K',\Sigma']$ is a signed clique.

Suppose now that $[K',\Sigma']$ is a signed clique and let $u$ and $v$ be any two vertices of $K$, with $u\neq z$ and $v\neq z$.
Since $u$ and $v$ are not identifiable in $[K',\Sigma']$, either $uv$ is an edge of $[K',\Sigma']$, and thus of $(K,D)$,
or $u$ and $v$ are antipodal vertices in some $UC_4$ of $[K',\Sigma']$, which implies that they are the end vertices of some $UP_3$
in $(K,D)$. In both cases, $u$ and $v$ are not identifiable in $(K,D)$, which implies that $(K,D)$ is a \tec{} clique.
\end{proof}

%Suppose first that $(K,D)$ is a \tec{} clique and
%let $u$ and $v$ be any two vertices of $K'$. 
%If $u=z$ or $v=z$, then there is an edge $uv$ by construction. 
%Otherwise, both $u$ and $v$ belong to $K$. 
%If there is no edge $uv$ in $K$, then $u$ and $v$ are end vertices of some $UP_3$ of $(K,D)$ (\resp antipodal vertices of some $UC_4$ of $[K',\Sigma']$). 
%	In the case where $[K',\Sigma']$ is a signed clique, let $w$ and $t$ be the vertices such that $uwvt$ is a $UC_4$ then by construction $uwv$ or $utv$ is a $UP_3$ in $(K,D)$.
%	In the case where $(K,D)$ is a \tec{} clique, let $w$ be the vertex such that $uwv$ is a $UP_3$ then by construction $uwvz$ is a $UC_4$ in $(K',D')$ and thus in $[K',\Sigma']$. 
%	Hence, we cannot identify $u$ and $v$ and since this statement holds for any pair of vertices of the signed graph (\resp \tec{} 

The problem %{\sc Achromatic number} %ACHROMATIC NUMBER
of deciding whether the achromatic number of a graph is at least $k$,
for some integer $k$,
has been shown to be NP-complete even when restricted to small classes of graphs
in~\cite{Yannakakis1980} and~\cite{Bodlaender1989}. We recall the definition of the problem and
these results below.

%\noindent
%Problem: ACHROMATIC NUMBER\\
%Instance:  a  graph $G$ and an integer $k$.\\
%Question: Is $\psi(G) \geq k$ ?

\PROBLEM{Achromatic number}{A graph $G$ and an integer $k$}{Is $\psi(G) \geq k$?}

\begin{thm}[Yannakakis and Gavril\cite{Yannakakis1980}]
  \label{np-complete-th0}
  The problem {\sc Achromatic number} is NP-complete even when restricted to complements of bipartite graphs.
\end{thm}

\begin{thm}[Bodlaender~\cite{Bodlaender1989}]
  \label{np-complete-th1}
  The problem {\sc Achromatic number} is NP-complete even when restricted to graphs which are simultaneously connected interval graphs and co-graphs.
\end{thm}

We will show that a number of problems related to the achromatic numbers
we have introduced are NP-complete.
We first define the following problem (the name in brackets is the acronym of the problem).

\PROBLEM{\tec{} graph achromatic number [2ec-an]}{A \tec{} graph $(G,C)$ and an integer $k$}
{Is $\psi_2(G,C) \geq k$?}

%\noindent
%Problem: TWO EDGE COLORED GRAPH ACHROMATIC NUMBER [2EC-AN]\\
%Instance:  a \tec{} graph $(G,C)$ and an integer $k$.\\
%Question: Is $\psi_2(G,C) \geq k$ ?
%
%\vspace{10pt}

Since any graph $G$ can be regarded as the \tec{} graph $(G,\varnothing)$,
and every complete coloring of $G$ as a complete coloring of $(G,\varnothing)$,
the problem {\sc 2ec-an} contains the problem {\sc Achromatic number}.
The following theorem thus easily follows from Theorems~\ref{np-complete-th0} and~\ref{np-complete-th1}.

\begin{thm}
  \label{np-complete-th2}
  The problem {\sc 2ec-an} is NP-complete even when restricted to graphs which are simultaneously connected interval graphs and co-graphs 
  or to complements of bipartite graphs.
\end{thm}

\begin{proof}
We can verify in polynomial type whether a coloring of $(G,C)$ on $q \geq k$ colors
is a proper complete coloring of $(G,C)$ or not. Thus {\sc 2ec-an} is in NP.
The result then follows by the above remark.
%from Theorems~\ref{np-complete-th0} and~\ref{np-complete-th1}
%since {\sc 2ec-an} contains {\sc Achromatic number}.
%  Moreover our problem contains the problem {\sc Achromatic number} for graphs which is NP-complete for connected interval graphs and co-graphs as shown in~\cite{Bodlaender1989} and  for graphs that are complements of bipartite graphs as shown in~\cite{Yannakakis1980}. Thus it is NP-hard.
\end{proof}

We now define all the other decision problems that we will consider.
%Now we will define the following problems based on our parameters: \vspace{10pt}

%Problem: SIGNED GRAPH ACHROMATIC NUMBER [SIGNED-AN]\\
%Instance:  a signed graph $(G,\Sigma)$ and an integer $k$.\\
%Question: Is $\psi_s(G,\Sigma) \geq k$ ?
%
%\noindent
%Problem: SIGNED GRAPH MAX ACHROMATIC NUMBER  [SIGNED-MAX-AN]\\
%Instance:  a signed graph $(G,\Sigma)$ and an integer $k$.\\
%Question: Is $\psi_{\max}(G,\Sigma) \geq k$ ?
%
%\noindent
%Problem: SIGNED GRAPH MIN ACHROMATIC NUMBER  [SIGNED-MIN-AN]\\
%Instance:  a signed graph $(G,\Sigma)$ and an integer $k$.\\
%Question: Is $\psi_{\min}(G,\Sigma) \geq k$ ?
%
%\noindent
%Problem: SIMPLE GRAPH MAX TWO EDGE COLORED ACHROMATIC NUMBER \sloppy [SG-MAX-2EC-AN]\\
%Instance: a graph $G$ and an integer $k$.\\
%Question: Is $\psi_{\max}(G) \geq k$ ?
%
%\noindent
%Problem: SIMPLE GRAPH MIN TWO EDGE COLORED ACHROMATIC NUMBER \sloppy [SG-MIN-2EC-AN]\\
%Instance: a graph $G$ and an integer $k$.\\
%Question: Is $\psi_{\min}(G) \geq k$ ?
%
%\noindent
%Problem: SIMPLE GRAPH MAX SIGNED ACHROMATIC NUMBER \sloppy [SG-MAX-SIGNED-AN]\\
%Instance: a graph $G$ and an integer $k$.\\
%Question: Is $\psi_{\max}^{signed}(G) \geq k$ ?
%
%\noindent
%Problem: SIMPLE GRAPH MIN SIGNED ACHROMATIC NUMBER \sloppy [SG-MIN-SIGNED-AN]\\
%Instance: a graph $G$ and an integer $k$.\\
%Question: Is $\psi_{\min}^{signed}(G) \geq k$ ?

\PROBLEM{Signed graph achromatic number [Signed-an]}
{A signed graph $[G,\Sigma]$ and an integer $k$}
{Is $\psi_s[G,\Sigma] \geq k$?}

\PROBLEM{Signed graph max-achromatic number [Signed-max-an]}
{A signed graph $[G,\Sigma]$ and an integer $k$}
{Is $\psi_{\max}[G,\Sigma] \geq k$?}

\PROBLEM{Signed graph min-achromatic number [Signed-min-an]}
{A signed graph $[G,\Sigma]$ and an integer $k$}
{Is $\psi_{\min}[G,\Sigma] \geq k$ ?}

\PROBLEM{Graph \tec{} max-achromatic number [Max-2ec-an]}
{A graph $G$ and an integer $k$}
{Is $\psi_{\max}(G) \geq k$?}

\PROBLEM{Graph \tec{} min-achromatic number [Min-2ec-an]}
{A graph $G$ and an integer $k$}
{Is $\psi_{\min}(G) \geq k$?}

\PROBLEM{Graph signed max-achromatic number [Max-signed-an]}
{A graph $G$ and an integer $k$}
{Is $\psi_{\max}^{signed}(G) \geq k$?}

\PROBLEM{Graph signed min-achromatic number [Min-signed-an]}
{A graph $G$ and an integer $k$}
{Is $\psi_{\min}^{signed}(G) \geq k$?}

Our main results are gathered in the two following theorems,  and will be proved in the next section.
%We will show in section~\ref{section-NP} the following theorem:

\begin{thm}
\label{np-complete-psis}
  The problem {\sc Signed-an} is NP-complete even when restricted to graphs which are simultaneously connected interval graphs and co-graphs 
  or to complements of bipartite graphs.
\end{thm}

\begin{thm}
  \label{np-complete-th3}
  \label{th-np}
  The following problems are NP-complete:
  \begin{itemize}
    \setlength\itemsep{0mm}
    %\item {\sc Signed-an}, even when restricted to connected perfect graphs
    \item {\sc Signed-max-an}, even when restricted to connected diamond-free perfect graphs
    \item {\sc Max-2ec-an}, even when restricted to connected diamond-free perfect graphs
    \item {\sc Max-signed-an}, even when restricted to connected perfect graphs.
  \end{itemize}
\end{thm}

For the three other problems it is easy to show that:

\begin{thm}
  \label{np-complete-th5}
  \label{np-complete-th7}
  \label{np-complete-th9}
The problems
{\sc Signed-min-an},
{\sc Min-2ec-an} and
{\sc Min-signed-an}
are in $\Pi_2$.
%  \begin{itemize}
%    \setlength\itemsep{0mm}
%    \item  SIGNED-MIN-AN
%    \item SG-MIN-2EC-AN
%    \item SG-MIN-SIGNED-AN
%  \end{itemize}
\end{thm}

A natural question is thus the following.

\begin{ques}
  Are the three problems of Theorem~\ref{np-complete-th5} $\Pi_2$-Complete?
\end{ques}

\begin{table}
\begin{center}
\begin{tabular}{|l|l|l|l|}
  \hline
   & Ordinary graphs & \tec{} graphs & Signed graphs \\
  \hline
  $\psi$ & NP-complete [Th~\ref{np-complete-th1}]& N.A. & N.A.\\
  \hline
  $\psi_2$ & N.A.& NP-complete [Th~\ref{np-complete-th2}] & N.A.\\
  \hline
  $\psi_s$ & N.A. & N.A. & NP-complete [Th~\ref{np-complete-psis}]\\
  \hline
  $\psi_{\max}$ & NP-complete [Th~\ref{np-complete-th3}] & N.A. & NP-complete [Th~\ref{np-complete-th3}]\\
  \hline
  $\psi_{\min}$ & $\Pi_2$ (complete ?)[Th~\ref{np-complete-th7}]  & N.A. & $\Pi_2$ (complete ?) [Th~\ref{np-complete-th5}]\\
  \hline
  $\psi_{\max}^{signed}$ & NP-complete [Th~\ref{np-complete-th3}] &  N.A. &  N.A.\\
  \hline
  $\psi_{\min}^{signed}$ & $\Pi_2$ (complete ?)  [Th~\ref{np-complete-th9}] &  N.A. &  N.A.\\
  \hline
\end{tabular}
\caption{Decision problems related to achromatic numbers.}
\label{tab:summary}
\end{center}
\end{table}

Table~\ref{tab:summary} summarizes our results and what is known on
decision problems related to achromatic numbers.
%Here is a summary of what we showed:
%
%\vspace{10pt}

\section{Proof of Theorems \ref{np-complete-psis} and~\ref{np-complete-th3}}
\label{section-NP}

In order to prove that all these problems are NP-complete, we need to prove that they are in NP and are NP-hard.
We first prove that the four problems belong to NP.

\begin{proof}[Proof of membership to NP]
Suppose that we have an instance of {\sc Signed-an} (\resp {\sc Signed-max-an}) consisting of a signed graph $[G,\Sigma]$ and an integer $k$.
Assume we are given a \tec{} graph $(G,C)$, a coloring
$\alpha$ of $(G,C)$.
We can verify that $\alpha$ is a complete coloring of $(G,C)$ (\resp of $[G,\Sigma]$ by choosing $(G,C)$ as the representative) using at least $k$ colors
and, that $(G,C) \in [G,\Sigma]$ in polynomial time as shown in \cite{harary1953}.
Therefore, both problems {\sc Signed-an} and {\sc Signed-max-an} are in NP.

Suppose now that we have an instance of {\sc Max-2ec-an} (\resp {\sc Max-signed-an}) consisting of an ordinary graph $G$ and an integer $k$.
Moreover, if we are given a signature $C$ and a vertex coloring
$\alpha$ of $G$, we can verify in polynomial time that $\alpha$ is a complete coloring of $(G,C)$ (\resp of $[G,\Sigma_C]$, the signed graph defined by $\Sigma_C := C$)  using at least $k$ colors, which implies that
both problems {\sc Max-2ec-an} and {\sc Max-signed-an} are in NP.
%
%If we have an instance of the first two cases, {\it i.e.} a signed graph \signedG{} and an integer $k$, then we can verify in polynomial time that $\psi_{\max}(G,\Sigma) \geq k $ or $\psi_s(G,\Sigma) \geq k$ depending on our problem.
%
%To do that it suffices to get a \tec{} graph on $G$: $(G,C)$, a coloring of $(G,C)$ and the list of switched vertices. The list of switched vertices allows us to verify that  $(G,C) \in (G,\Sigma)$ and then we can just check that the coloring is proper and complete (\tec{} complete or signed complete depending on the case) on more than $k$ colors.
%
%For the last two problems, it suffices to give us a 2-edge coloring of $G$ and a vertex coloring of $G$, we can check that it is proper and complete in polynomial time.
%
%Thus all four problems are in NP.
\end{proof}

We are now ready to prove Theorem~\ref{np-complete-psis}.

\begin{proof}[Proof of Theorem \ref{np-complete-psis}]
We already showed that the problem is in NP.
Take now an instance of {\sc Achromatic number} consisting of a connected graph $G$ and an integer $k$. 
Let $H$ be the graph obtained from $G$ by adding a vertex $z$ such that, for all $u \in V(G)$, $zu \in E(H)$. If $G$ is an interval graph and a co-graph then so is $H$ by construction. Indeed the interval corresponding to $z$ can be chosen as the convex union of the intervals of the other vertices and there is no induced $P_4$ containing $z$.  If $G$ is a complement of a bipartite graph than so is $H$, we just add an isolated vertex to the complement of $G$, this graph is still bipartite and its complement is $H$. In both cases, $H$ is in the relevant subclass.
We claim that $\psi(G) \geq k$ if and only if $\psi_s[H,\varnothing] \geq k + 1$. 

Suppose first that $\psi(G) = p \geq k$. This means that there exists a surjective homomorphism from $G$ to $K_p$, the complete graph on $p$ vertices. 
By applying this homomorphism on the copy of $G$ in $H$, we get $[H,\varnothing] \rightarrow_s [K_{p+1},\varnothing]$. 
Hence, $\psi_s[H,\varnothing] \geq k + 1$.

Suppose now that $\psi_s[H,\varnothing] \geq k+1$.
There exists a signed homomorphism from $[H,\varnothing]$ to $[K,\Pi]$, a signed clique on at least $k+1$ vertices. 
We can re-sign $[K,\Pi]$ in such a way that, in its canonical representative, the vertex $z$ is only incident with 
positive edges and has not been re-signed by the homomorphism. 
We want to show that, in this signature of $[K,\Pi]$, all the edges non incident with $z$ are positive. 
Suppose to the contrary that there is a negative edge which is the image of the edge $uv$ of $H$. 
Then, exactly one of $u$ or $v$ has been re-signed by the homomorphism, but in this case the edge linking this vertex to $z$ would be negative, a contradiction. 
Thus, all the edges are positive, which gives that $K$ is a complete graph and no vertices have been re-signed. 
If we take the restriction of our homomorphism to $G$, then we get a homomorphism from $G$ to a complete graph of size at least $k$.

If the graph $G$ was in one of the two subclasses of the theorem then we get our result.
\end{proof}

In order to prove Theorem~\ref{np-complete-th3},
we will use a reduction from
the following decision problem.

%The proof of Theorem~\ref{np-complete-th3} relies on a reduction from:

\PROBLEM{3-partition}
{A set $A = \set{a_1, \dots, a_{3m}} \in \NN^{3m}$ and an integer $B$ such that $\frac{B}{4} < a_i < \frac{B}{2}$
for every $i$, $1\leq i \leq m$}
{Is there a partition $\set{P_1, \dots P_m}$ of $A$  such that $\abs{P_i} = 3$ and $\sum\limits_{a_j \in P_i} a_j = B$
for every $i$, $1\leq i \leq m$?}

%\noindent
%Problem: 3-PARTITION\\
%Instance: $A = \set{a_1, \dots, a_{3m}} \in \NN^{3m}$ and an integer $B$ such that $\forall i\leq 3m, \frac{B}{4} < a_i < \frac{B}{2}$.\\
%Question: Is there a partition of $A$, $\set{P_1, \dots P_m}$  such that $\forall 1 \leq  i \leq m$, $\abs{P_i} = 3$ and $\sum\limits_{a_j \in P_i} a_j = B$ ?\\

This problem has been shown to be strongly NP-complete in~\cite{GareyJohnson1990} by Garey and Johnson.
Note that we can multiply each $a_i$ and $B$ by $m+1$, and thus assume $m < a_i$ for all $i$, $1\leq i\leq 3m$.
Also note that in a positive instance of {\sc 3-partition}, we have
\begin{equation}\label{eq:mB}
\sum\limits_{1 \leq i \leq 3m} a_i = mB.
\end{equation} 
We will assume that equation (\ref{eq:mB}) holds in the rest of this section.
Moreover, since {\sc 3-partition} is strongly NP-complete, the size of the instance can be taken as $O(Bm)$.

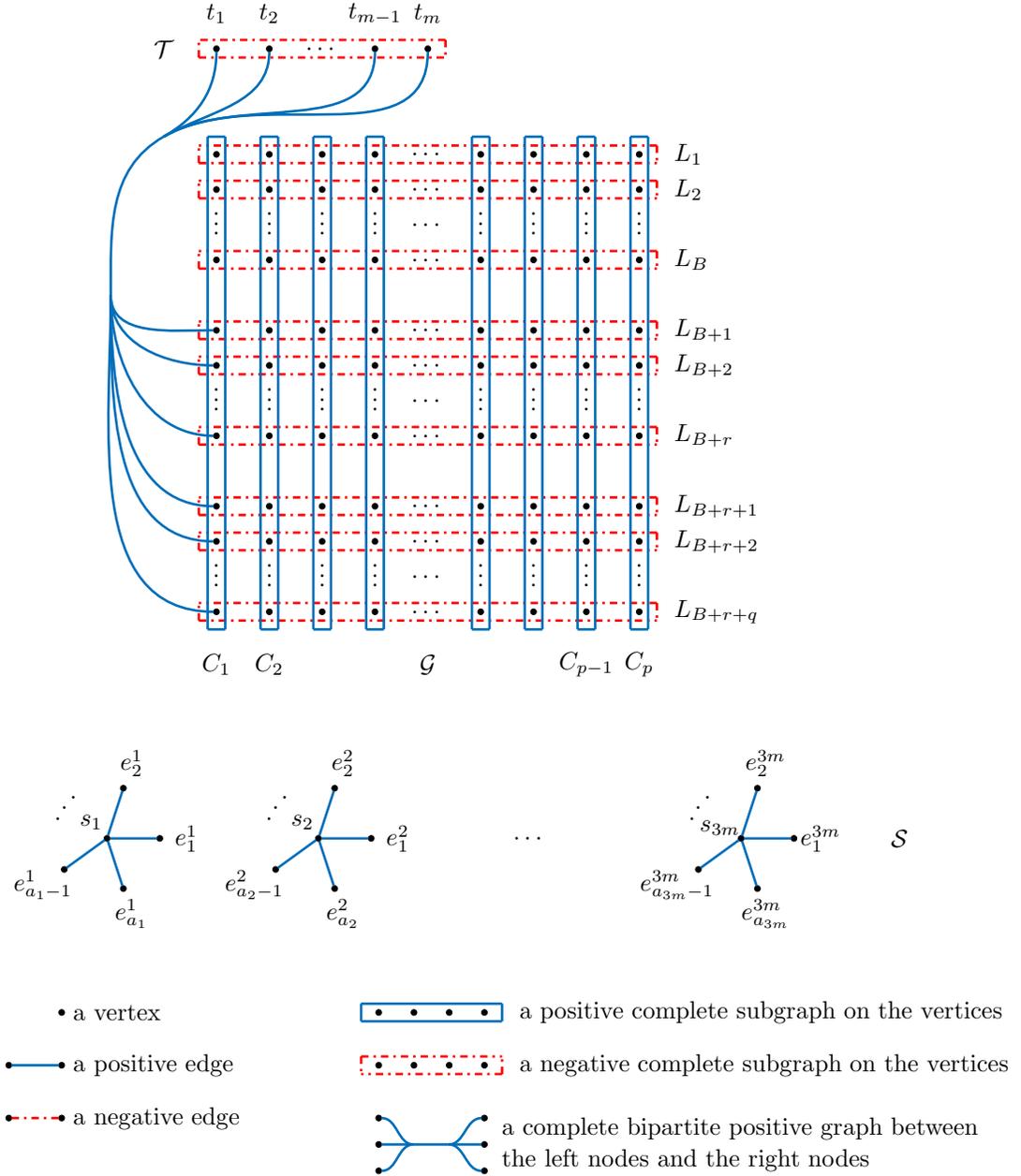
\begin{figure}%[ht]
\centering

\begin{center}
\subfloat{
\centering
\begin{tikzpicture}[scale=1.5]
\tikzstylemacro{}

\foreach \i in {0,1,2,3,5,6,7,8}{
  \foreach \j in {0}{
    \node[n] (q\i\j) at (\i/2,\j/3)  {};}
  \node (q\i2) at (\i/2,1/3 + 1/12)  {$\vdots$};
  \foreach \j in {2,3}{
    \node[n] (q\i\j) at (\i/2,\j/3)  {};}

  \foreach \j in {0}{
    \node[n] (r\i\j) at (\i/2,1+2/3+\j/3)  {};}
  \node (r\i2) at (\i/2,1+3/3 + 1/12)  {$\vdots$};
  \foreach \j in {2,3}{
    \node[n] (r\i\j) at (\i/2,1+2/3+\j/3)  {};}

  \foreach \j in {0}{
    \node[n] (b\i\j) at (\i/2,3+1/3+\j/3)  {};}
  \node (b\i2) at (\i/2,3+2/3 + 1/12)  {$\vdots$};
  \foreach \j in {2,3}{
    \node[n] (b\i\j) at (\i/2,3+1/3+\j/3)  {};}
}

\foreach \j in {0,...,3}{
  \node (q7\j) at (4/2,\j/3)  {$\dots$};}
\foreach \j in {0,...,3}{
  \node (r7\j) at (4/2,1+2/3 + \j/3)  {$\dots$};}
\foreach \j in {0,...,3}{
  \node (b7\j) at (4/2,3+1/3 + \j/3)  {$\dots$};}

\foreach \j in {0,1,3,4}{
  \node[n] (t\j) at (\j/2,5 + 1/3)  {};}

\node (t2) at (2/2,5 + 1/3)  {$\dots$};

\foreach \j in {0,2,3,5,7,8,10,12,13}{
  \coordinate (cn\j1) at (-1/6,\j/3-1/12);
  \coordinate (cn\j2) at (-1/6,\j/3+1/12);
  \coordinate (cn\j3) at (4+1/6,\j/3-1/12);
  \coordinate (cn\j4) at (4+1/6,\j/3+1/12);
  \draw (cn\j1) edge[r] (cn\j2) edge[r] (cn\j3);
  \draw (cn\j4) edge[r] (cn\j2) edge[r] (cn\j3);
}

\foreach \i in {0,1,2,3,5,6,7,8}{
  \coordinate (cp\i1) at (\i/2-1/12,0-1/6);
  \coordinate (cp\i2) at (\i/2+1/12,0-1/6);
  \coordinate (cp\i3) at (\i/2 - 1/12,4+3/6);
  \coordinate (cp\i4) at (\i/2 + 1/12,4+3/6);
  \draw (cp\i1) edge[g] (cp\i2) edge[g] (cp\i3);
  \draw (cp\i4) edge[g] (cp\i2) edge[g] (cp\i3);
}

\coordinate (tr1) at (-1/6,5+1/3-1/12);
\coordinate (tr2) at (-1/6,5+1/3+1/12);
\coordinate (tr3) at (4/2+1/6,5+1/3-1/12);
\coordinate (tr4) at (4/2+1/6,5+1/3+1/12);
\draw (tr1) edge[r] (tr2) edge[r] (tr3);
\draw (tr4) edge[r] (tr2) edge[r] (tr3);

\coordinate (i1) at (-1,3);
\foreach \j in {0,2,3}{
  \draw (i1) edge[g,in=180, out=270] (q0\j);
  \draw (i1) edge[g,in=180, out=270] (r0\j);
}

\coordinate (i2) at (-0.5,4.5);
\foreach \j in {0,1,3,4}{
  \draw (i2) edge[g,in=270, out=30] (t\j);
}

\draw (i2) edge[g,in=90, out=210] (i1);

\node () at (-0.5,5+1/3) {$\mathcal{T}$};
\node () at (4/2,-0.5) {$\mathcal{G}$};

\node () at (0,5+4/6) {$t_1$};
\node () at (1/2,5+4/6) {$t_2$};
\node () at (3/2,5+4/6) {$t_{m-1}$};
\node () at (2,5+4/6) {$t_m$};

\node () at (0,-0.5) {$C_1$};
\node () at (0.5,-0.5) {$C_2$};
\node () at (3.5,-0.5) {$C_{p-1}$};
\node () at (4,-0.5) {$C_p$};

\node[anchor=west] () at (4.25,4+1/3) {$L_1$};
\node[anchor=west] () at (4.25,4) {$L_2$};
\node[anchor=west] () at (4.25,3+1/3) {$L_B$};

\node[anchor=west] () at (4.25,8/3) {$L_{B+1}$};
\node[anchor=west] () at (4.25,2+1/3) {$L_{B+2}$};
\node[anchor=west] () at (4.25,5/3) {$L_{B+r}$};

\node[anchor=west] () at (4.25,3/3) {$L_{B+r+1}$};
\node[anchor=west] () at (4.25,2/3) {$L_{B+r+2}$};
\node[anchor=west] () at (4.25,0) {$L_{B+r+q}$};

\end{tikzpicture}
}
\end{center}

\begin{center}
\subfloat{
\centering
\begin{tikzpicture}[scale=1.5]
\tikzstylemacro{}

\foreach \j in {0,1}{
  \node[n](s\j) at (2*\j,-2) {};
  \foreach \i in {0,1,3,4}{
    \node[n](s\j\i) at ($(2*\j,-2)+(360*\i/5:0.5cm)$) {};
    \draw (s\j) edge[g] (s\j\i);
  }
  \node (s\j1d) at ($(2*\j,-2)+(360*2.2/5:0.5cm)$) {$.$};
  \node (s\j2d) at ($(2*\j,-2)+(360*2/5:0.5cm)$) {$.$};
  \node (s\j3d) at ($(2*\j,-2)+(360*1.8/5:0.5cm)$) {$.$};
}
\foreach \j in {4}{
  \node[n](s\j) at (2*\j-2,-2) {};
  \foreach \i in {0,1,3,4}{
    \node[n](s\j\i) at ($(2*\j-2,-2)+(360*\i/5:0.5cm)$) {};
    \draw (s\j) edge[g] (s\j\i);
  }
  \node (s\j1d) at ($(2*\j-2,-2)+(360*2.2/5:0.5cm)$) {$.$};
  \node (s\j2d) at ($(2*\j-2,-2)+(360*2/5:0.5cm)$) {$.$};
  \node (s\j3d) at ($(2*\j-2,-2)+(360*1.8/5:0.5cm)$) {$.$};
}
\node (s2) at (4,-2) {$\dots$};

\node () at ($(2*0,-2) + (-0.15,0.15)$) {$s_1$};
\node () at ($(2*1,-2) + (-0.15,0.15)$) {$s_2$};
\node () at ($(2*3,-2) + (-0.2,0.1)$) {$s_{3m}$};

\foreach \j in {1,2}{
  \node () at ($(2*\j-2,-2)+(360*0/5:0.75cm)$) {$e^{\j}_1$};
  \node () at ($(2*\j-2,-2)+(360*1/5:0.75cm)$) {$e^{\j}_2$};
  \node () at ($(2*\j-2,-2)+(360*3/5:0.75cm)$) {$e^{\j}_{a_\j -1}$};
  \node () at ($(2*\j-2,-2)+(360*4/5:0.75cm)$) {$e^{\j}_{a_\j}$};
}

\foreach \j in {4}{
  \node () at ($(2*\j-2,-2)+(360*0/5:0.75cm)$) {$e^{3m}_1$};
  \node () at ($(2*\j-2,-2)+(360*1/5:0.75cm)$) {$e^{3m}_2$};
  \node () at ($(2*\j-2,-2)+(360*3/5:0.75cm)$) {$e^{3m}_{a_{3m} -1}$};
  \node () at ($(2*\j-2,-2)+(360*4/5:0.75cm)$) {$e^{3m}_{a_{3m}}$};
}

\node () at (7.5,-2) {$\mathcal{S}$};

\end{tikzpicture}
}
\end{center}

\begin{center}
\subfloat{
\begin{tikzpicture}[scale=1.5]
\tikzstylemacro{}

\node[n,label=right:a vertex] () at (0,0) {};
\node[n] (ep1) at (-0.5,-0.5) {};
\node[n, label=right: a positive edge] (ep2) at (0,-0.5) {};
\draw (ep1) edge[g] (ep2);
\node[n] (en1) at (-0.5,-1) {};
\node[n, label=right: a negative edge] (en2) at (0,-1) {};
\draw (en1) edge[r] (en2);

\foreach \j in {0,1,2,3}{
  \node[n] (cg\j) at (3+\j/3,0) {};
}
\coordinate(g1) at (3-1/6,1/12) {};
\coordinate(g2) at (3-1/6,-1/12) {};
\coordinate(g3) at (4+1/6,1/12) {};
\coordinate(g4) at (4+1/6,-1/12) {};
\draw (g1) edge[g] (g2) edge[g] (g3);
\draw (g4) edge[g] (g2) edge[g] (g3);
\node[label=right:a positive complete subgraph on the vertices] () at (4+1/6,0) {};

\foreach \j in {0,1,2,3}{
  \node[n] (cg\j) at (3+\j/3,-0.5) {};
}
\coordinate(g1) at (3-1/6,1/12-0.5) {};
\coordinate(g2) at (3-1/6,-1/12-0.5) {};
\coordinate(g3) at (4+1/6,1/12-0.5) {};
\coordinate(g4) at (4+1/6,-1/12-0.5) {};
\draw (g1) edge[r] (g2) edge[r] (g3);
\draw (g4) edge[r] (g2) edge[r] (g3);
\node[label=right:a negative complete subgraph on the vertices] () at (4+1/6,-0.5) {};

\node[n] (t11) at (3,-1) {};
\node[n] (t12) at (3,-1.25) {};
\node[n] (t13) at (3,-1.5) {};

\node[n] (t21) at (4,-1) {};
\node[n] (t22) at (4,-1.25) {};
\node[n] (t23) at (4,-1.5) {};

\coordinate (i1) at (3.33,-1.25) {};
\coordinate (i2) at (3.66,-1.25) {};
\foreach \j in {1,2,3}{
  \draw (i1) edge[g,in=0,out=180] (t1\j);
  \draw (i2) edge[g,in=180,out=0] (t2\j);
}
\draw (i1) edge[g] (i2);
\node[label=right:a complete bipartite positive graph between] () at (4,-1.1) {};
\node[label=right:the left nodes and the right nodes] () at (4,-1.4) {};

\end{tikzpicture}
}
\end{center}
\caption{The signed graph $[H(\mathcal{I}),\Sigma(\mathcal{I})]$ and the legend for notation.}
\label{np-tec-construt}
\end{figure}

Given an instance $\mathcal{I}$ of {\sc 3-partition}, we will construct the signed graph
$[H(\mathcal{I}),\Sigma(\mathcal{I})]$ (see Figure~\ref{np-tec-construt}).
This graph is composed of three parts.
\begin{enumerate}
\item
The subgraph $\mathcal{S}$, called the ``stars'', contains $3m$ stars $S_1,\dots,S_{3m}$.
Each star $S_i$ has a center vertex, $s_i$, and $a_i$ leaves $e^i_1,\dots, e^i_{a_i}$.

\item
The subgraph $\mathcal{T}$, called the ``target'',
is a negative clique on $m$ vertices $t_1,\dots,t_m$.

\item
The subgraph $\mathcal{G}$, called the ``grid'', contains $p(B+r+q)$ vertices,
where $p= 2\binom{m}{2} +2m(B+r+q)+ q+1$, $q=6m+2Bm$ and $r = \binom{m}{2} + qm +1$.
These vertices are denoted $x_{i,j}$, for $1\leq i\leq B+q+r$ and $1\leq j\leq p$.
There is a positive edge between $x_{i,j}$ and $x_{k,\ell}$ if and only if $i =k$ and $j \neq \ell$.
There is a negative edge between $x_{i,j}$ and $x_{k,\ell}$ if and only if $j =\ell$ and $i \neq k$.
We denote the columns and rows of $\mathcal{G}$
respectively by $C_j = \set{x_{i,j}\ | \ 1\leq i \leq B+r+q}$ and $L_i = \set{x_{i,j}\ | \ 1\leq j \leq p}$.

\end{enumerate}

Finally we add a positive edge $t_{\ell}x_{i,1}$
for every $\ell$, $1\leq\ell\leq m$, and every $i$, $B+1 \leq i \leq B+r+q$.

The signature of $H(\mathcal{I})$ above described
(one of its many equivalent signatures) is the easiest one to work with.
%We describe one of the many signature of the signed graph but this one will be easier to work with.
Note that $H(\mathcal{I})$ is a diamond-free perfect graph but is not connected.
We will conduct the proof without the connectivity requirement and then explain how to modify it
in such a way that the considered graph is connected.

From now on, we let $k(\mathcal{I}) := m + p(B+r+q)$.

\begin{claim}
\label{claim:Npfirstimplication}
	If $\mathcal{I}$ is an instance of {\sc 3-PARTITION} that admits a solution then, for the canonical representative $(H(\mathcal{I}),C_{\Sigma(\mathcal{I})})$ of $[H(\mathcal{I}),\Sigma(\mathcal{I})]$, we have:
	 $$k(\mathcal{I}) \leq \psi_2(H(\mathcal{I}),C_{\Sigma(\mathcal{I})}) \leq \psi_{\max}[H(\mathcal{I}),\Sigma(\mathcal{I})] \leq \psi_{\max}(H(\mathcal{I})).$$
\end{claim}

\begin{proof}
If there is a solution $\set{P_1,\dots,P_m}$ to $\mathcal{I}$ then,
on the \tec{} graph $(H(\mathcal{I}), C_{\Sigma(\mathcal{I})})$, we
assign one color to each vertex of $\mathcal{T} \cup \mathcal{G}$ (which gives $k(\mathcal{I})$ colors), then we 
identify $t_i$ and $s_j$ for $j\in P_i$ and every $i$, $1\leq i\leq m$. 
We thus get that each $t_i$ has $B$ positive neighbours of degree $1$ which correspond to the $B$ leaves of the three stars $S_j$ for $j \in P_i$.
For $1 \leq i \leq m$, we identify each of these $B$ neighbours of $t_i$
with a unique vertex in $\set{x_{1,1}, \dots, x_{B,1}}$.

We obtain a \tec{} clique on $k(\mathcal{I})$ colors. 
Indeed, if $u$ and $v$ are two vertices of the graph after the identification, we have three cases to consider. 
If $u$ and $v$ are both vertices of the target $\mathcal{T}$, then there is an edge $uv$ in the graph by construction of $\mathcal{T}$.
If $u$ and $v$ are both vertices of the grid $\mathcal{G}$, say $u = x_{i,j}$ and $v = x_{k,\ell}$, then $ux_{k,j}v$ is a $UP_3$ by construction of $\mathcal{G}$.
Otherwise, suppose that $u$ is a vertex of the target and $v$ is a vertex of the grid, say $u = t_i$ and $v= x_{j,k}$. 
If $j > B$ then, by construction of $[H(\mathcal{I}),\Sigma(\mathcal{I})]$, $ux_{j,1}v$ is a $UP_3$ (or $uv$ is an edge if $k=1$). In the other case, there is a positive edge $ux_{j,1}$ by the previous identifications and a negative edge $x_{j,1}x_{j,k}$ (if $k\neq1$) by construction of $\mathcal{G}$.
Hence, we cannot identify $u$ and $v$, which implies that the graph is a clique on $k(\mathcal{I})$ vertices.

We also need to prove that the homomorphism is well defined, {\ie} does not create any loop and does not create any digon. 
This follows from the fact that we identify edges of $\mathcal{S}$ with non-edges of $\mathcal{T} \cup \mathcal{G}$. 
\end{proof}

We want to prove that if $\psi_\max(H(\mathcal{I}),\Sigma(\mathcal{I})) \geq k(\mathcal{I})$ then $\mathcal{I}$ has a solution. There are two main ideas in the following result. 

First, we know that if $\psi_\max(H(\mathcal{I}),\Sigma(\mathcal{I})) \geq k(\mathcal{I})$, then the homomorphism that reaches the signed max-achromatic number creates a clique on more than $k(\mathcal{I})$ vertices. 
This clique has diameter~$2$, as indicated in Observation \ref{obs:2ecclique}. 
We want to show that constructing a ``large'' graph of diameter~$2$ from $H(\mathcal{I})$ implies that $\mathcal{I}$ has a solution.

Secondly, in this setting, this means that the identification performed to create the ``large'' graph of diameter~$2$, is similar to the one we did in the proof of Claim~\ref{claim:Npfirstimplication}.

\begin{lem}
  \label{np-lemma-tec}
Let $\mathcal{I}$ be an instance of {\sc 3-partition}.
%For any set of edges $C\subseteq E(H(\mathcal{I}))$, %two-edge coloring of $H(\mathcal{I})$, $(H(\mathcal{I}),C)$,
%if $\psi_2(H(\mathcal{I}),C) \geq k(\mathcal{I})$ then $\mathcal{I}$ admits a solution.
If there is a surjective homomorphism $\varphi$ from $H(\mathcal{I})$ to an ordinary graph $K$ of order greater than $k(\mathcal{I})$ and diameter at most $2$, then $\mathcal{I}$ has a solution. 
\end{lem}

The proof of Lemma \ref{np-lemma-tec} works as follows. We  first prove that each vertex of $K$ has a pre-image in $\mathcal{T}$ or $\mathcal{G}$. We then prove that the edges in $\mathcal{S}$ were identified in a  way similar to the construction in Claim~\ref{claim:Npfirstimplication}.

\begin{proof}
Let $\varphi$ be a homomorphism
of  $H(\mathcal{I})$ to $K$, an ordinary graph of order greater than $k(\mathcal{I})$ and diameter at most $2$. 

Let then
$$\alpha = \varphi(\mathcal{S}) \setminus \varphi(\mathcal{G} \cup  \mathcal{T}),\
\beta = \varphi(\mathcal{T}) \setminus \varphi(\mathcal{G})\
\mbox{and}\ \gamma = \varphi(\mathcal{G}).$$
The set $\alpha$ represents the vertices of $K$ that come from the identification of vertices only in $\mathcal{S}$. 
The set $\beta$ represents the vertices of $K$ that come from the identification of vertices in $\mathcal{T}$ or $\mathcal{S}$. 
The set $\gamma$ represents all other vertices.
Our first goal is to prove that the set $\alpha$ is empty.

Note that $(\alpha,\beta,\gamma)$ is a partition of $\varphi(H(\mathcal{I}))$,
and thus 
$$\abs{\alpha} +\abs{\beta} + \abs{\gamma} \geq k(\mathcal{I}) = m + p(B+r+q).$$
Moreover, we have 
$$\abs{V(H(\mathcal{I}))} - k(\mathcal{I}) = 3m + Bm.$$

Let $d = 3m + Bm$.
Since the homomorphism $\varphi$ can do only $d$ identifications of vertices,
there are at most $2d = q$ vertices in $H(\mathcal{I})$ which have been identified with some other vertex.
We denote by $\mathrm{Id}$ the set of vertices that were identified to another and by  $\mathrm{Id}_{\mathcal{G}}$
the set of vertices of $\mathcal{G}$ that have been identified with another vertex of $\mathcal{G}$.
%this set of vertices.
% of $\mathcal{G}$ that were identified to another vertex of $\mathcal{G}$.

Let $\mathcal{L} = \set{L_i\ |\ L_i \cap \mathrm{Id}_{\mathcal{G}}= \varnothing}$.
The set $\mathcal{L}$ is the set of lines (themselves sets of vertices) of the grid $\mathcal{G}$ that do not contain a vertex identified with another vertex of the grid.
Moreover, for every vertex $u \in K \setminus \gamma$ (\ie $u$ is a vertex that is not in the image of the grid), let
$$\mathcal{N}_u = \set{L_i \in \mathcal{L}\ |\ \varphi(L_i) \cap N_K(u) \neq \varnothing}.$$
The set $\mathcal{N}_u$ is the set of lines of the grid $\mathcal{G}$ that do not contain any vertices identified with another vertex of the grid and intersect the neighbourhood of $u$ in $K$. We claim that $\abs{\mathcal{N}_u} = \abs{\mathcal{L}}$. By definition of $\mathcal{N}_u$, we have $\abs{\mathcal{N}_u} \geq \abs{\mathcal{L}}$.

Suppose to the contrary that there exists $u \in K \setminus \gamma$ such that  $\abs{\mathcal{N}_u} < \abs{\mathcal{L}}$. Therefore,
there exists $L_i \in \mathcal{L}$ with $\varphi(L_i) \cap N_K(u) = \varnothing$.
Since K has diameter at most~$2$, for every vertex $v \in \varphi(L_i)$,
there exists a neighbour $w$ of $u$ that is a neighbour of $v$ (recall that there is no edge $uv$).
There are at least $p-q$ vertices of $L_i$ belonging to some $C_j$ with $C_j \cap \mathrm{Id} = \varnothing$ with $1\leq j\leq p$ (\ie columns where we did not identified any vertices).
Among these columns,
there are at most $\abs{N_K(u)}$ $C_j$'s which contain a neighbour of $u$ in $\varphi(H(\mathcal{I}))$.
Thus, there are at least $p - q - \abs{N_K(u)}$ vertices of $L_i$ such that the column $C_j$ they belong to
intersects neither $\mathrm{Id}$ nor $\varphi^{-1}(N_K(u))$.
These vertices correspond to $p - q - \abs{N_K(u)}$ vertices in $\varphi(L_i)$ as they
do not belong to $\mathrm{Id}_{\mathcal{G}}$.

Moreover, 
$$\abs{N(u)} \leq \abs{E(H(\mathcal{I}))\setminus E(\mathcal{G})} \leq \binom{m}{2} + m(B+r+q).$$
%
%Therefore, there are at least  $p-q-\binom{m}{2}-m(B +r +q)$ vertices in $\varphi(L_j)$ whose column $\varphi(C_i)$ doesn't intersect $\mathrm{Id}_{\mathcal{G}}$ nor contains a neighbor of $u$. For each of these vertices, neither there neighbors in $\varphi(L_j)$ nor their neighbors in the column$\varphi(C_i)$ they belong to is a neighbor of $u$.
%
%Hence, each of these vertices is linked to $u$ by a path of length 2 whose interior vertex is in $\varphi(H(\mathcal{I})) \setminus \gamma$.
Therefore, there are at least $p - q - \binom{m}{2} - m(B+r+q)$ vertices in $L_i$ not belonging to a
column %$C_i$ such that $\varphi(C_i)$
whose image by $\varphi$ intersects $\mathrm{Id}$ or contains a neighbour of $u$.
It follows that every such vertex, say $x_{i,j}$, has no neighbour in $\varphi(L_i)\cup \varphi(C_j)$ which is a neighbour of $u$. 
As $L_i$ does not intersect $\mathrm{Id}_{\mathcal{G}}$ and $C_j$ does not intersect $\mathrm{Id}$, in $K$, $\varphi(L_i)\cup \varphi(C_j)$ contains all the neighbours of $x_{i,j}$ among the vertices of $\mathcal{G}$.
These remarks imply that $x_{i,j}$ is linked to $u$ by a path of length~2
whose interior vertex is in $K \setminus \gamma$ (\ie $x_{i,j}$ is linked to $u$ by a path which is not induced in the grid).
%It follows that none of these vertices have a neighbor in $\varphi(L_j)\cup \varphi(C_i)$ which is a neighbor of $u$,
%which implies that each of these vertices is linked to $u$ by a path of length 2
%whose interior vertex is in $\varphi(H(\mathcal{I})) \setminus \gamma$.
%

There are at most $\binom{m}{2} + m(B+r+q)$ edges not in $\gamma$, and $p-q-\binom{m}{2} -m(B+r+q) \geq \binom{m}{2} + m(B+r+q) +1$ vertices to which $u$ must be linked in $\varphi(L_i)$ by a path of length~2 whose interior vertex is not in $\gamma$.
This is not possible since we do not have enough edges that can be used for such paths. % all of these vertices.
Therefore, $\abs{\mathcal{N}_u} \geq \abs{\mathcal{L}} \geq B+r+q -\abs{\mathrm{Id}_{\mathcal{G}}} \geq B+r$.

Moreover, 
$$\abs{E(\varphi(H(\mathcal{I}))\setminus \gamma))} \leq Bm + \binom{m}{2} + (r+q) m - (r+q)(m-\abs{\beta}).$$
The last term comes from the fact that the edges between vertices of $\mathcal{G}$ and
$\mathcal{T} \setminus \varphi^ {-1}(\beta)$ have images by $\varphi$ in $\varphi(\mathcal{G})$ and do not contribute
to $\abs{E(\varphi(H(\mathcal{I}))\setminus \gamma))}$.
This number of edges must be greater than
$$\sum\limits_{u\in \alpha \cup \beta} \abs{\mathcal{N}_u} \geq (B+r)(\abs{\alpha} + \abs{\beta}).$$
But $\abs{\alpha} + \abs{\beta} \geq k(\mathcal{I}) - \abs{\gamma} \geq m$ by the choice of $k(\mathcal{I})$.
We then get 
$$Bm + \binom{m}{2} + (r+q)\abs{\beta} \geq Bm + r(\abs{\alpha} + \abs{\beta}),$$ and thus,
since $\abs{\beta} \leq m$, 
$$\binom{m}{2} + qm \geq r\abs{\alpha}.$$
If $\abs{\alpha} > 0$, then $\binom{m}{2} + qm + 1= r \leq  \binom{m}{2} + qm$,
a contradiction, and thus $\alpha = \varnothing$.

Because $\alpha = \varnothing$, every vertex of $\mathcal{S}$ is identified with a vertex of $\mathcal{G} \cup \mathcal{T}$, and since there are $d$ vertices in $\mathcal{S}$, this accounts for all the identifications.
We then get $\beta = \mathcal{T}$ and $\gamma = \mathcal{G}$,
which implies $\mathrm{Id}_{\mathcal{G}}= \varnothing$ and $\abs{\mathcal{N}_u} = B+q+r$.
The number of edges that can contribute to $\sum\limits_{u\in \beta} \abs{\mathcal{N}_u}$ is limited:
$Bm$ edges in $\mathcal{S}$ and $(r+q)m$ edges between $\mathcal{T}$ and $\mathcal{G}$,
which gives $\sum\limits_{u\in\beta} \abs{\mathcal{N}_u} \leq m(B+r+q)$.
Therefore, there is a one-to-one correspondence between the pairs in $\mathcal{T} \times \mathcal{L}$ and the edges that can contribute to the sum.

Suppose now that some $s_i$ was identified with a vertex in $\mathcal{G}$.
Since $m<a_i$, and no two leaves of a star can be identified with each over, the leaves of this star cannot all be identified with the vertices of $\mathcal{T}$ as $\mathcal{T}$ is of order $m$. So at least one leaf $e^i_j$ is identified with a vertex of $\mathcal{G}$, but this means that the edge $s_ie^i_j$ does not contribute to the sum $\sum_{u \in \beta} \abs{\mathcal{N}_u}$,
a contradiction.
%whose doesn't increase one $\mathcal{N}_u$ for $u\in \beta$

Now, note that for any $u \in \beta$, $\mathcal{N}_u \leq B + r+ q$. If at most two star centers are identified with some vertex $u$ of $\mathcal{T}$, then, since these two stars have less than $B$ leaves between them, we have $\mathcal{N}_u < B + r+ q$ and thus 
$\sum_{u \in \beta} \mathcal{N}_u < m(B + r+ q)$.
%Suppose that more than four $s_i$'s have been identified with the same vertex $u$ of $\mathcal{T}$.
%Then, at least one of the edges of the corresponding four stars has been identified with a vertex $v \in \mathcal{G}$ which was already a neighbor of $u$.
%This contradicts the one-to-one correspondence between the pairs in $\mathcal{T} \times \mathcal{L}$ and the edges that can contribute to the sum.
%
Hence, we finally get that each vertex $u\in\mathcal{T}$ was identified with three $s_i$'s whose sum
of subscripts equals $B+r+q - (q +r) = B$.
This gives us a partition of the set $A$ which is a solution of {\sc 3-partition}.
\end{proof}

We can now prove that {\sc Max-2ec-an} and {\sc Signed-max-an} are NP-complete.

\begin{proof}[Proof that {\sc Max-2ec-an} (\resp {\sc Signed-max-an}) is NP-complete]\mbox{}\\
We already proved that both these problems are in NP.
If $\mathcal{I}$ is an instance of {\sc 3-PARTITION}, then we construct the graph $H(\mathcal{I})$ (\resp the signed graph $[H(\mathcal{I}),\Sigma(\mathcal{I})]$) in polynomial time.

By Claim \ref{claim:Npfirstimplication}, if $\mathcal{I}$ has a solution, 
then $k(\mathcal{I}) \leq \psi_\max(H(\mathcal{I}))$ (\resp $k(\mathcal{I}) \leq \psi_\max[H(\mathcal{I}),\Sigma(\mathcal{I})]$). 

If $k(\mathcal{I}) \leq \psi_\max(H(\mathcal{I}))$ (\resp $k(\mathcal{I}) \leq \psi_\max[H(\mathcal{I}),\Sigma(\mathcal{I})]$), then there exists a signature $C$ and a surjective \tec{} homomorphism  $\varphi$ such that $(H(\mathcal{I}),C) \rightarrow_2 (K,D)$ by $\varphi$, where $(K,D)$ is a \tec{} clique of order greater than $k(\mathcal{I})$ and, in the signed case, $C \in \Sigma(\mathcal{I})$.
As $(K,D)$ has diameter~$2$, by Lemma \ref{np-lemma-tec}, we get that $\mathcal{I}$ has a solution since $\varphi$ is also a surjective homomorphism from $H(\mathcal{I})$ to $K$.

The problems {\sc Max-2ec-an} and {\sc Signed-max-an}  are thus NP-complete even when restricted to diamond-free perfect graphs.
To make the graph $H(\mathcal{I})$ connected, it suffices to increase $q$ by one and $r$ by $3m$,
and to add an edge joining the vertices $x_{B+r+q,1}$ and $s_i$ for every $i$, $1\leq i\leq 3m$.
The graph is now clearly connected and the same arguments as in Lemma~\ref{np-lemma-tec} works,
since  the $3m$ new edges cannot be used to create conflicts between vertices of $\mathcal{T}$ and $\mathcal{G}$
that did not already exist.
\end{proof}

We now consider the case of signed graphs.
Let $H'(\mathcal{I})$ be the  graph
obtained from $H(\mathcal{I})$
(the underlying graph of the signed graph previously defined, see Figure~\ref{np-tec-construt}) by adding a new vertex $z$ such that,
for every vertex $v \in V(H(\mathcal{I}))$, $zv$ is an edge.
% of $(H'(\mathcal{I}), \Sigma'(\mathcal{I}))$ when the chosen signature of $H(\mathcal{I})$ is $ \Sigma(\mathcal{I})$.
We also define $k'(\mathcal{I}) = k(\mathcal{I}) +1$.

We are left to prove that {\sc Max-signed-an} is NP-complete.

\begin{proof}[Proof that {\sc Max-signed-an} is NP-complete]\mbox{}\\
We already proved that this problem is in NP.
If $\mathcal{I}$ is an instance of {\sc 3-PARTITION}, then we construct the graph $H'(\mathcal{I})$ in polynomial time. Note that $H'(\mathcal{I})$ is a connected perfect graph.

By Claim \ref{claim:Npfirstimplication}, if $\mathcal{I}$ has a solution, then
$(H(\mathcal{I}),C_{\Sigma(\mathcal{I})}) \rightarrow_2 (K,D)$, where $(K,D)$ is a \tec{} clique of order greater than $k(\mathcal{I})$ and $C_{\Sigma(\mathcal{I})} = \Sigma(\mathcal{I})$. 
Thus, $(H'(\mathcal{I}),C_{\Sigma(\mathcal{I})})  \rightarrow_2 (K',D')$, where $(K',D')$ is obtained from $(K,D)$ by adding one vertex $z$ that is a positive neighbour of every vertex of $(K,D)$. 
By Lemma \ref{2ectosignedclique}, $[K',\Sigma']$, where $\Sigma' = D'$, is a signed clique. 
Hence 
 $k'(\mathcal{I}) \leq \psi_\max^{signed}(H'(\mathcal{I}))$. 

If $k'(\mathcal{I}) \leq \psi_\max^{signed}(H'(\mathcal{I}))$, then there exists a signature $\Sigma_1$ and a surjective signed homomorphism  $\varphi'$ such that $[H'(\mathcal{I}),\Sigma_1] \rightarrow_s [K',\Pi']$ by $\varphi'$, where $[K',\Pi']$ is a signed clique of order greater than $k'(\mathcal{I})$. 
Up to re-signing $[K',\Pi']$, we can assume that $z$ is a positive neighbour of all the other vertices.
Let $K$ be the graph obtained from $K'$ by removing the image of $z$. 
Note that $z$ was not identified by $\varphi'$. 
By Lemma \ref{2ectosignedclique}, $(K,D)$ is a \tec{} clique, where $D$ is $\Sigma_1$ from which 
we removed the edges incident to $z$. 
Let $\varphi$ be the restriction of $\varphi'$ to $H(\mathcal{I})$. 
Then, by $\varphi$, $H(\mathcal{I}) \rightarrow_2 K$.
As $[K',\Pi']$ is a signed clique, $K$ has diameter $2$ and, by Lemma \ref{np-lemma-tec}, 
we get that $\mathcal{I}$ has a solution.
\end{proof}

\section{Discussion}
\label{section-discussion}

In this paper, we introduced and study achromatic numbers of \tec{} graphs and of signed graphs.
In particular, Theorems \ref{np-complete-th2}, \ref{np-complete-psis} and~\ref{th-np} state that computing the achromatic number of a \tec{} graph or of a signed graph  is NP-complete.

The two following results allow to conclude that the problem of computing the achromatic number of an ordinary graph is FPT.
Recall that a {\it reducing congruence class} (an {\it \rc class} for short) on a graph $G$ is an equivalence class
of the relation~$\equiv_G$ defined by $u \equiv_G v$ if and only if $N_G(u)=N_G(v)$,
where $N_G(u)$ denotes the neighborhood of the vertex $u$ in $G$.
In other words, $u \equiv_G v$ if and only if $u$ and $v$ are twins in $G$.

\begin{thm}[Hell and Miller\cite{Hell1976}, Hoffman~\cite{Hoffman1972} and M\'at\'e~\cite{Mate1981}]
There is a computable function $f:\NN \rightarrow\NN$ such that,
for every integer $k$, if a graph $G$ has more than $f(k)$ \rc classes then $\psi(G) \geq k$.
\label{th:Hell-Miller}
\end{thm}

\begin{thm}[Farber, Hahn, Hell and Miller~\cite{Farber1986}]
For a fixed integer $k$, there is an algorithm that, given a graph $G$, determines whether $\psi(G) \geq k$ or not
in time $O(|E(G)|)$.
\label{th:Farber-Hahn-Hell-Miller}
\end{thm}

The following question is thus natural when considering these two results.

\begin{ques}
 Is it possible to determine if one of our parameters is greater than some integer $k$ in FPT time where $k$ is the parameter?
\end{ques}

Theorem~\ref{th:Hell-Miller} can be generalized to \tec{} graphs and to signed graphs but we were not able to generalize Theorem~\ref{th:Farber-Hahn-Hell-Miller} using the same techniques as in~\cite{Farber1986}.

\bigskip

\noindent
{\bf Acknowledgment.}
I would like to thank Herv\'e Hocquard and \'Eric Sopena for their advice and for their help in the writing of this article. Also thanks to Pascal Ochem for helpful discussions.

\bibliographystyle{plain}
\bibliography{biblio}

\end{document}